\documentclass[12pt,oneside]{article}

\usepackage{amssymb}

\usepackage{amsmath}

\usepackage{amsfonts}

\usepackage{longtable}

\newtheorem{theorem}{Theorem}[section]
\newtheorem{lemma}[theorem]{Lemma}

\newtheorem{corollary}[theorem]{Corollary}

\newtheorem{definition}[theorem]{Definition}

\newtheorem{example}[theorem]{Example}

\newtheorem{fact}[theorem]{Fact}

\newtheorem{remark}[theorem]{Remark}

\newenvironment{proof}
{\begin{trivlist}  \item \textsc{Proof:}~} {\hfill $\Box$
\end{trivlist}}

\newenvironment{proof of claim}
{\begin{trivlist}  \item \textsc{Proof of Claim:}~} {\hfill $\Box$
\end{trivlist}}

\newenvironment{subclaim}
{\begin{trivlist}  \item \textsc{Subclaim:}~} {\end{trivlist}}
\newenvironment{proof of subclaim}
{\begin{trivlist}  \item \textsc{Proof of subclaim:}~} {\hfill $\Box$ (SUBCLAIM)
\end{trivlist}}

\newenvironment{proof of claim1}
{\begin{trivlist}  \item \textsc{Proof of Claim 1.}~} {\hfill $\Box$ (CLAIM 1)
\end{trivlist}}

\newenvironment{proof of claim2}
{\begin{trivlist}  \item \text{Proof of Claim 2.}~} {\hfill $\Box$ (CLAIM 2)
\end{trivlist}}

\newenvironment{proof of claim3}
{\begin{trivlist}  \item \textsc{Proof of Claim 3.}~} {\hfill $\Box$ (CLAIM 3)
\end{trivlist}}

\newcommand{\closure}[1]{\ensuremath{\mathrm{cl}}(#1)}
\newcommand{\interior}[1]{\ensuremath{\mathrm{int}}(#1)}

\def \Def {\operatorname{Def}}

\def \Q{\mathbb{Q}}

\def\Ind#1#2{#1\setbox0=\hbox{$#1x$}\kern\wd0\hbox to 0pt{\hss$#1\mid$\hss}
\lower.9\ht0\hbox to 0pt{\hss$#1\smile$\hss}\kern\wd0}

\def\Notind#1#2{#1\setbox0=\hbox{$#1x$}\kern\wd0\hbox to 0pt{\mathchardef
\nn=12854\hss$#1\nn$\kern1.4\wd0\hss}\hbox to
0pt{\hss$#1\mid$\hss}\lower.9\ht0 \hbox to
0pt{\hss$#1\smile$\hss}\kern\wd0}

\newcommand{\domain}[1]{\ensuremath{\mathrm{domain}}(#1)}

\newcommand{\ma}{\mathfrak{m}}

\newcommand{\bk}{\mathbf{k}}

\begin{document}

\title{Measuring definable sets in o-minimal fields}
\author{Jana Ma\v{r}\'ikov\'{a} and Masahiro Shiota}
\maketitle

\begin{abstract}
We introduce a non real-valued measure on the definable sets contained in the finite part of a cartesian power of an o-minimal field $R$.  The measure takes values in an ordered semiring, the Dedekind completion of a quotient of $R$.  We show that every measurable subset of $R^n$ with non-empty interior has positive measure, and that the measure is preserved by definable $C^1$-diffeomorphisms with Jacobian determinant equal to $\pm 1$.

\end{abstract}

\begin{section}{Introduction}

Let $R$ be an o-minimal field, i.e. an o-minimal expansion of a real closed field.
In \cite{nip}, Hrushovski, Peterzil and Pillay ask, roughly, the following question:
Let $B[n]$ be the lattice of all bounded $R$-definable subsets of $R^n$.  Define an equivalence relation $\sim $ on $B[n]$ as follows: $X \sim Y$ if modulo a set of dimension $<n$ we have $\phi (X)=Y$ for some definable $C^1$-diffeomorphism $\phi$ with absolute
value of the determinant of the Jacobian of $\phi$ at $x$ equal to 1 for all $x\in X$.   Suppose $X\in B[n]$ is of dimension $n$.  Is there a finitely additive map $\mu \colon B[n] \to \mathbb{R}^{\geq 0}\cup \{ \infty \}$ which is $\sim$-invariant and such that $\mu X\in \mathbb{R}^{>0}$?

 Note that for cardinality reasons it is impossible to find a real-valued measure that would assign a real non-zero value to every bounded definable set with non-empty interior in some big o-minimal field.

We remark that the answer to the question posed in \cite{nip} is yes if $R$ is pseudo-real\footnote{Let $\mathcal{L}$ be an expansion of the language of ordered rings, and let $T(\mathcal{L})$ be the collection of all $\mathcal{L}$-sentences true in all $\mathcal{L}$-expansions of the reals.  A structure is called pseudo-real if it is a model of $T(\mathcal{L})$.} in the sense of van den Dries (\cite{lou-qu}):
If there is an o-minimal field $\mathcal{S}$ (in the language $\mathcal{L}$) for which the answer to the question posed in \cite{nip} is no, then we can find definable bounded sets $X,Y \subseteq \mathcal{S}^n$ and a positive integer $m$ so that $X\not\sim \emptyset$ and $(m+1) X \dot\cup Y \sim mX$, where $(m+1)X$ is the disjoint union of $m+1$ copies of $X$ (see \cite{nip}, Proposition 5.5, p. 576).  But this fact is expressible by a parameter-free first-order sentence in $\mathcal{L}$, and this sentence is false in all $\mathcal{L}$-expansions of the reals, hence our structure is not pseudo-real.

While the framework of o-minimality was developed with a view towards structures on the reals (see Shiota \cite{shiotabook} and van den Dries \cite{book}), it is well-known that not all o-minimal structures are pseudo-real.  More concretely, Lipshitz and Robinson show in \cite{lr} that the field of Puiseux series $\bigcup_n \mathbb{R}((t^{\frac{1}{n}}))$ in $t$ over $\mathbb{R}$ expanded by functions given by overconvergent power-series (henceforth the L-R field) is o-minimal, and Hrushovski and Peterzil show in \cite{uk} that the L-R field is not pseudo-real.

Let $V$ be the convex hull of $\mathbb{Q}$ in $R$.  Then $V$ is a convex subring of $R$, hence a valuation ring.  Let $\pi \colon V\to \bk$ be the corresponding residue/standard part map.  The corresponding residue field $\bk$ is the ordered real field $\mathbb{R}$ if $R$ is at least $\omega$-saturated.
In \cite{bo}, Berarducci and Otero define a measure on the lattice $SB[n]$ of all strongly bounded definable subsets of $R^n$, i.e. the definable subsets of $V^n.$  Assuming that $R$ is at least $\omega$-saturated, one way to define the Berarducci-Otero measure is to assign to $X\in SB[n]$ the Lebesgue measure of $\pi X$.  It was shown in \cite{real} that the Berarducci-Otero measure is $\sim$-invariant, which yields a partial answer to the question posed in \cite{nip}: The answer is yes whenever the set $X\in B[n]$ in question is contained in $V^n$, and $\pi X$ has non-empty interior.  However, the Berarducci-Otero measure assigns zero to every set whose standard part has empty interior. 

In this paper we drop the requirement of the measure being real-valued.  More precisely, we define a map $\mu \colon SB[n]\to \widetilde{V}$, where $\widetilde{V}$ is an ordered semiring, such that for all $X,Y \in SB[n]$, $\mu (X\dot\cup Y) = \mu X + \mu Y$, and $\mu X>0$ iff the interior of $X$ is nonempty (see Theorem \ref{maintheorem}).
The underlying set of $\widetilde{V}$ is constructed as the Dedekind completion of a quotient of $V^{\geq 0}$.  
The construction of the measure itself resembles the construction of Lebesgue measure.  Taking a quotient of $V^{\geq 0}$ serves the purpose of identifying lower and upper measures.  For measurable sets whose standard part has non-empty interior our measure agrees with the Berarducci-Otero measure.  In fact, the unique minimal ring that embeds $\widetilde{V}$ is $\mathbb{R}$.  On the collection of strongly bounded definable sets whose standard part has empty interior, $\mu$  resembles a dimension function:  
There, we have $\mu (X\dot\cup Y)=\max \{ \mu X , \mu Y \}$, and if $\mu X < \mu Y$, then $X$ can be isomorphically embedded (in the sense of \cite{nip}) into finitely many copies of $Y$ (this follows from Lemmas \ref{box} and \ref{boxlemma}).  We do not know if the strict inequality above can be replaced by a nonstrict one.

We show that $\mu$ has the analogue of the invariance property defined in \cite{nip}: Suppose $X, Y\subseteq V^n$ are definable and $\phi \colon U \to V$ is a definable $C^1$-diffeomorphism with $X\subseteq_0 U$, $Y \subseteq_0 Y$ and $|J\phi (x)| =1$, where $J\phi (x)$ is determinant of the Jacobian of $\phi$ at $x$.  Then $\mu X = \mu Y$ (see Corollary \ref{invariance}).

In the case of the L-R field we can modify the definition of $\mu$ to obtain a finitely additive measure on all of $B[n]$.  This measure takes values in the Dedekind completion of the value group of the standard valuation.  It agrees with $\mu$ for sets $X\in SB[n]$ so that $\interior{\pi X}=\emptyset$, but assigns the same value to all sets $X \in SB[n]$ with $\interior{\pi X}\not=\emptyset$.

\medskip\noindent
We thank Michel Coste and Marcus Tressl for their advice.  The first author whishes to thank the second author for his hospitality during a visit to Nagoya, Japan.  

\end{section}

\begin{section}{Notation and conventions}
The letters $k, l, m,n$ denote non-negative integers.

Let $M$ be a structure.  Then {\em $M$-definable\/} (or simply {\em definable\/}, if $M$ is clear from the context) means definable in the language of $M$, with parameters from $M$.
We denote by $\Def^n (M)$ the collection of all $M$-definable subsets of $M^n$.

We fix $V$ to be the convex hull of $\Q$ in $R$.  Then $V$ is a convex subring of $R$, hence a valuation ring, with residue (standard part) map $\pi \colon V \to \bk$, maximal ideal $\ma$, and (ordered) residue field $\bk$.  For $X\subseteq R^n$ we set $\pi X=\pi (X \cap V^n )$.  We denote by $v$ the corresponding valuation $R\to \Gamma \cup \{ \infty \}$, where $\Gamma = R^{\times}/(V \setminus \ma )$ is the (divisible ordered abelian) value group.


Let $M$ be an o-minimal structure.  For $k< n$ we denote by $p^{n}_{k}$ the projection map $M^n \to M^k$ given by $x\mapsto (x_1 , \dots ,x_k )$.  If $Y \subseteq M^n$ is definable and non-empty and $x\in M^n$, then $$d(x,Y):=\inf \{d(x,y):\; y\in Y \},$$ where $d(x,y)$ is the euclidean distance between $x$ and $y$.
For $X, Y\subseteq M^n$ we write $X\subseteq_0 Y$ if $\dim{(X\setminus Y)}<n$, and 
$X=_0 Y$ if $X\subseteq_0 Y$ and $Y \subseteq_0 X$.
If $f\colon X\to M$, where $X\subseteq M^n$, is a function, then $$\Gamma f:=\{ (x,y):\; x\in X \mbox{ and } f(x)=y \}$$ is the graph of $f$.

For $X\subseteq R$ we set $X^{\geq r}:=\{ x\in X\colon \, x\geq r  \}$.  The sets $X^{\leq r}$, $X^{<r}$, and $X^{>r}$ are defined similarly.  If $Y$ is another subset of $R$, then $X^{>Y}$ is the set $$\{ x\in X: x>y \; \mbox{ for all } y \in Y  \}.$$  The set $X^{<Y}$ is defined similarly.

A box in $R^n$ is a set of the form $[a_1 , b_1 ]\times \dots \times [a_n ,b_n ]$, where $a_i < b_i$ and $a_i , b_i \in R^{>0}$.

If $X\subseteq M^n$, then $\closure{X}$ denotes the closure of $X$ and $\interior{X}$ denotes the interior of $X$ with respect to the interval topology on $M$.
\end{section}

\begin{section}{The set of values $\widetilde{V}$}
In this section we define the set of values $\widetilde{V}$ of our measure, and we show that it can be equipped with the structure of an ordered semiring.

First, we define an equivalence relation $\sim$ on $V^{\geq 0}$.
\begin{definition}\label{equivalence}
Let $x,y \in V^{\geq 0}$.  Then $x\sim y$ if either
\begin{itemize}
\item both $x$ and $y$ are in $\ma^{\geq 0}$, and $$y^q \leq x \leq y^p \mbox{ for all } p,q\in \Q^{>0}, \; p<1 , \; q>1, \mbox{ or}$$
\item both
$x$ and $y$ are $>\ma$, and $\pi x = \pi y$.
\end{itemize}
\end{definition}
Note that the ordering $\leq$ on $R$ induces an ordering $\leq$ on $V^{\geq 0}/\sim $.  For $x\in V^{\geq 0}$ we denote by $[x]$ the $\sim$-equivalence class of $x$.

\smallskip\noindent
In the next definition a {\em Dedekind cut in $V^{\geq 0}/\sim$\/} is the union of a downward closed subset of $V^{\geq 0}/\sim$ without a greatest element with the set $V^{<0}/\sim$, where $\sim$ is extended to $V^{<0}$ by setting $x\sim y$ iff $-x \sim -y$, for $x,y \in V^{<0}$.

\begin{definition}\label{operations}
We let $\widetilde{V}$ be the collection of all Dedekind cuts in $V^{\geq 0}/ \sim$.  
We define an ordering $ \leq $ and binary operations $+$ and $\cdot$ on $\widetilde{V}$ as follows. 
Let $X,Y \in \widetilde{V}$.  Then
\begin{itemize}
\item[a)] $X \leq Y \mbox{ iff } \forall x\in \bigcup X \;\; \exists y \in \bigcup Y$ with $x\leq y$.
\item[b)] $X+Y :=\{ x+y:\, x \in \bigcup X \, \& \, y\in \bigcup Y \}/\sim$.
\item[c)] $X\cdot Y:=\{ x \cdot y:\, x \in \bigcup X^{\geq 0} \, \& \, y\in \bigcup Y^{\geq 0} \}/\sim \cup  \; V^{<0}/\sim $.
\end{itemize}
For $a\in V^{\geq 0}$ we denote by $\widetilde{a}$ the cut $$\{ [x] :\; x\in V^{\geq 0} \mbox{ and }[x] <[a]  \} \cup \; V^{<0}/\sim .$$ 
\end{definition}
Next, we show that + and $\cdot$ are well-defined, and that $\sim$ is a congruence.  The lemma below is used throughout the paper without explicit reference.





\begin{lemma}
Let $x,y \in \ma^{> 0}$ and suppose $v (x)=v (y)$.  Then $x\sim y$.
\end{lemma}

\begin{proof}
First note that
$x\sim nx$ for all $n$: If $p\in \mathbb{Q}^{>0}$, $p<1$, then $$v(x^p ) = p \cdot v(x)<v(x)=v(nx),$$ hence $nx\leq x^p$.

Now assume $x<y$ (the other cases are similar).
Since $v (x)=v (y)$, we have $\frac{y}{x} < n$ for some $n$.  Hence $x<y<nx$, and so $x\sim y$.
\end{proof}

\begin{remark}
\em We do not have $x\sim y$ iff $v(x)=v(y)$ on $\ma^{>0}$.  To see this assume that $R$ is $\omega$-saturated, let $x\in \ma^{>0}$, and let $y$ be any element realizing the type $p(z)$ consisting of all formulas $nx < z < x^p$, where $n=1,2,\dots $ and $p$ ranges over all positive rationals $<1$.  Then $x\sim y$ but $v(x)\not=v(y)$.  
\end{remark}

\begin{lemma}\label{add}
Let $X,Y \in \widetilde{V}$.  Then $X+Y \in \widetilde{V}$.
\end{lemma}
 \begin{proof}
It is clear that $X+Y$ is downward closed and contains $V^{<0}/\sim$.  It is left to show that it does not have a greatest element.  Let $x\in \bigcup X$ and $y\in \bigcup Y$.  We may assume $x\leq y$.

If $y >\ma$, take $y' \in \bigcup Y$ so that $[y] \; < [y']$.  Then $|y-y'|>\ma$, so $(x+y) - (x+y')>\ma$, hence $[x+y] <[x+y']$.

So suppose $y\in \ma^{> 0}$.  Let $y'\in \bigcup Y$ be such that $y<y^p <y'$ for some $p\in \mathbb{Q}^{>0}$ with $p<1$.  Then $$v(x+y)=v(y)>p\cdot v(y)\geq v(y')=v(x+y'),$$ and so $[x+y] <  [x+y'] $ because $y\not\sim y^p$.

The case when $Y=\widetilde{0}$ is clear.  
\end{proof}

\begin{lemma}\label{addcong}
Let $x,y \in V^{\geq 0}$.  Then $\widetilde{x}+\widetilde{y}=\widetilde{x+y}$.
\end{lemma}
\begin{proof}
We may assume that $x\leq y$.  It suffices to show that if $x'\sim x$ and $y'\sim y$, then $x'+y'\sim x+y$, and if $z\sim x+y$, then there are $x'\sim x$ and $y'\sim y$ so that $z=x'+y'$.  The cases when $y=0$ and when $y>\ma$ are clear.

So suppose $y\in \ma^{> 0}$. If $x'\sim x$ and $y'\sim y$, then $v(x'+y')=v(y')$ and $v(x+y)=v(y)$, so $$x'+y' \sim y' \sim y \sim x+y.$$  If $z\sim x+y$, then, since $v(x+y)=v(y)$, we have $z\sim y$, and so $x'=x$ and $y'=z-x$ work.
\end{proof}

\begin{lemma}\label{mult}
Let $X,Y \in \widetilde{V}$.  Then $X\cdot Y \in \widetilde{V}$.
\end{lemma}
\begin{proof}
It is clear that $X\cdot Y$ is a downward closed subset of $V/\sim$ containing $V^{<0}/\sim$.  It is left to show that $X\cdot Y$ does not have a greatest element.
The case when there is $x\in (\bigcup X)^{>\ma}$ and $y\in (\bigcup Y)^{>\ma}$ is clear, as is the case when $X=\widetilde{0}$ or $Y=\widetilde{0}$.

So suppose $x\in \bigcup X$ and $y\in \bigcup Y$ and assume $x\leq y$.  If $x\in \ma^{>0}$ and $y>\ma$, then  $[xy] < [x'y]$ for any $x'\in \bigcup X$ with $[x] < [x']$.
If $x\in \ma^{>0}$ and $y\in \ma^{>0}$, then we can find $p \in \mathbb{Q}^{>0}$, $p<1$ so that $x<x^p<x'$ and $y<y^p<y'$ for some $x'\in \bigcup X$ and $y'\in \bigcup Y$.  Then $xy<x^p y^p <x'y'$, hence $[xy] <[x'y']$.
\end{proof}

\begin{lemma}\label{multcong}
Let $x,y \in V^{\geq 0}$.  Then $\widetilde{x}\cdot \widetilde{y} = \widetilde{xy}$.
\end{lemma}

\begin{proof}
We may assume that $x\leq y$.
It suffices to show that if $x'\sim x$ and $y'\sim y$, then $x'y' \sim xy$, and if $z\sim xy$, then there are $x'\sim x$ and $y'\sim y$ so that $z\sim x'y'$.  It is easy to check that the lemma holds if $x,y>\ma$ or if $x=0$.  

So suppose $x\in \ma^{>0}$, and let $x'\sim x$ and $y'\sim y$.  If $y>\ma$, then $v(x'y')=v (x')$ and $v (xy)=v (x)$, hence $x'y'\sim x'\sim x\sim xy$.  If $y \in \ma$, then $x'y'\sim xy$ is immediate from the definition of $\sim$.

Now let $z\sim xy$ and assume $xy<z$.  It suffices to prove that $x\sim \frac{z}{y}$ (as then $z=\frac{z}{y}\cdot y \in \bigcup X\cdot Y$).  Assume towards a contradiction that this is not the case.  Then, as $x<\frac{z}{y}$, we would have $x^{p}<\frac{z}{y}$ for a positive rational $p<1$.  Moreover, since $xy\sim z$, we have $z\leq x^q y^q$ for all positive rationals $q<1$.  Thus $yx^p <x^q y^q$ for all $q<1$, $q\in \mathbb{Q}^{>0}$.  Then $x^{p-q} <y^{q-1}$ for all $q<1$, $q\in \mathbb{Q}^{>0}$.  For $q=\frac{p+1}{2}<1$ we obtain $x^{\frac{p}{2}-\frac{1}{2}}<y^{\frac{p}{2}-\frac{1}{2}}$, where $\frac{p}{2}-\frac{1}{2}<0$ (as $p<1$), a contradiction with $x\leq y$.

The case when $z\sim xy$ and $z<xy$ is handled similarly and left to the reader.
\end{proof}

\smallskip\noindent
From now on we shall assume that $R$ is $\omega$-saturated, in order to have $\bk =\mathbb{R}$.  This is no loss of generality: By Theorem 3.3 in \cite{cj}, for any elementary extension $R'$ of $R$, the structure $(R',V')$, where $V'$ is the convex hull of $\mathbb{Q}$ in $R'$, is an elementary extension of $(R,V)$.
\begin{remark}\label{rem}
\em
\begin{itemize}
\item It is now easy to check that $(\widetilde{V}, \leq , +, \cdot , \widetilde{0}, \widetilde{1})$ is an ordered semiring. 
\item The Dedekind completion of $V^{>\ma}/\sim$ is $\mathbb{R}^{>0}$.  We shall thus feel free to identify this part of $\widetilde{V}$ with $\mathbb{R}^{>0}$. For $a \in \mathbb{R}^{>0}$ we shall sometimes write $\widetilde{a}$ to indicate that $a$ is viewed as an element of $\widetilde{V}$.  Since $R$ is $\omega$-saturated, for any  $a\in \mathbb{R}^{>0}$, $\widetilde{a}=\widetilde{r}$ for some $r\in V^{>\ma}$.


\item Let $X,Y \in \widetilde{V}$.  
\begin{trivlist}


\item[i)] If $X \in \mathbb{R}^{>0}$ and $Y \not\in \mathbb{R}^{>0}$, then $X+Y =  X$.

\item[ii)] If $X\not\in \mathbb{R}^{>0}$ and $Y \not\in \mathbb{R}^{>0}$, then $X+Y = \max\{ X,Y  \}$.

\end{trivlist}

\item
We could extend
Definition \ref{equivalence} to all of $R^{\geq 0}$ by setting $x\sim y$ iff $x^{-1}\sim y^{-1}$ for $x,y \in R^{>V}$, and the set of all Dedekind cuts in $R^{\geq 0}/\sim $ could be made into an ordered semiring similarly as in Definition \ref{operations}.  However,
 $\sim$ is not a congruence with respect to $\cdot$ when considered as an equivalence relation on $R^{\geq 0}$.  To see this, consider the product of $\epsilon$ and $\frac{1}{\epsilon}$ for $\epsilon \in \ma^{> 0}$.  We have $\widetilde{\epsilon \cdot \frac{1}{\epsilon}} = \widetilde{1}$, but $(n\epsilon ) \sim \epsilon $, hence $n \in \bigcup \widetilde{\epsilon}\cdot \widetilde{\frac{1}{\epsilon }}$ for all $n=1,2,\dots$.
This would force us to assign to the box $[0,\epsilon ]\times [0,\frac{1}{\epsilon}]$  measure $>\widetilde{n}$ for all $n$.  In general,
this problem cannot be fixed by identifying all of $\widetilde{R}\cap \mathbb{R}^{>0}$: Let $a,b \in \ma^{>0}$, $a<b$, be such that $a\sim b$ but $v(a)\not=v(b)$.  Then there is $c\in R^{>V}$ with $\widetilde{c} < \widetilde{a}\cdot \widetilde{\frac{1}{b}}=\widetilde{1}$.

The special case when $v(a)=v(b)$ iff $a\sim b$ for all $a,b \in \ma^{\geq 0}$ will be dealt with in the last section of this paper.
\end{itemize}
\end{remark}

\end{section}

\begin{section}{Measuring definable subsets of $[0,1]^n$}
In this section, we define the lower and upper measures of definable sets contained in $[0,1]^n$, and we show that they conincide.  This yields a measure on the definable subsets of $[0,1]^n$ which is then extended to a measure on the definable subsets of $V^n$ in Section 5.

We shall consider the structure $\mathbb{R}_0$, which has as underlying set $\mathbb{R}$, and whose basic relations are the sets $\pi X$, where $X\in \Def^n R$ for some $n$.  As a weakly o-minimal structure on the reals, $\mathbb{R}_0$ is necessarily o-minimal.  We shall use the facts below; the first one is  Proposition 5.1, p. 188, in \cite{real}, the second one is extracted from the proof of Lemma 2.15, p. 124, in \cite{thesispaper}, and the third is Corollary 2.5, p. 120 in \cite{thesispaper}.

\begin{fact}\label{closed}
Let $X\in \Def^n (\mathbb{R}_0 )$.  Then there is $Y \in \Def^n (R)$ so that $\pi Y = \closure{X}$.
\end{fact}
\begin{fact}\label{intersection}
Let $X,Y \in \Def^n (R)$ be non-empty.  Then there is $\epsilon \in \ma^{>0}$ so that $\pi (X\cap Y^{\epsilon })=\pi X \cap \pi Y$, where $Y^{\epsilon}=\{ x\in R^n :\; d(x,Y)\leq \epsilon \}$.
\end{fact}
\begin{fact}\label{vbox}
Let $X \in \Def^n (R)$, and suppose $\interior{\pi X}\not=\emptyset$.  Then there is a box $B\subseteq X$ with $\interior{\pi B} \not=\emptyset$.
\end{fact}

\begin{definition}\label{defmeasure}
\begin{enumerate}
\item Let $X \subseteq [0,1]^{n}$ be an $(i_1 , \dots ,i_n )$-cell.  We define the {\em lower measure\/} $\underline{\mu}$ and {\em upper measure\/} $\overline{\mu}$ of $X$ by induction on $n$.
\begin{enumerate}
 \item If $X$ is a $(0)$-cell, then $\underline{\mu}X=\overline{\mu}X = 0$.  If $X=(a,b)$ where $a<b$, then $$\underline{\mu}X=\overline{\mu}X=\widetilde{b-a} \in \widetilde{R} .$$

\item Suppose $\underline{\mu}X$ and $\overline{\mu}X$ have been defined for $(i_1 , \dots ,i_n )$-cells.  If $X$ is an $(i_1 , \dots , i_{n+1})$-cell so that $i_j = 0$ for some $j\in \{ 1,\dots ,n+1  \}$, then $\underline{\mu}X=\overline{\mu}X=0$.  If $X=(f,g)$ is an $(i_1 ,\dots , i_{n+1} )$-cell so that $i_j =1$ for all $j\in \{ 1,\dots , n+1   \}$, then set $h=g-f$ and define
$\underline{\mu}X$ to be the 
supremum of 
$$\sum_{i=1}^{k} \widetilde{z}_{i-1} \cdot \underline{\mu} (h^{-1} [z_{i-1}, z_{i}]) $$ as $k\to \infty$ and $z_0 , \dots ,z_k $ range over all elements of $[0,1]_R$ with $$0=z_0 < \dots <z_k =1.$$ 
The upper measure $\overline{\mu} X$ is defined to be
the infimum of 
$$ \sum_{i=1}^{k}\widetilde{z}_i \cdot \overline{\mu} (h^{-1} [z_{i-1}, z_i ]) $$ as $k\to \infty$ and $z_0 , \dots ,z_k$ range over all elements of $[0,1]_R$ with $$0=z_0 < \dots <z_k =1.$$ 
\end{enumerate}

\item
Let $X\subseteq [0,1]^n$ be definable, and let $\mathcal{D}$ be a decomposition of $R^n$ into cells that partitions $X$.  Suppose $X=D_1 \cup D_2 \cup \dots \cup D_k$, where all $D_i \in \mathcal{D}$.  Then $\underline{\mu}_{\mathcal{D}} X= \sum_{i=1}^{k}\underline{\mu} D_i$ and $\overline{\mu}_{\mathcal{D}} X= \sum_{i=1}^{k}\overline{\mu} D_i$.

\end{enumerate}
\end{definition}
We shall also refer to the sum
$$\sum_{i=1}^{k} \widetilde{z}_{i-1} \cdot \underline{\mu} (h^{-1} [z_{i-1}, z_{i}]) $$ in the definition above as the {\em lower sum\/} of $f$ corresponding to the partition $\{z_0 , \dots ,z_k \}$, and to the sum $$ \sum_{i=1}^{k}\widetilde{z}_i \cdot \overline{\mu} (h^{-1} [z_{i-1}, z_i ]) $$ as the {\em upper sum\/} of $f$ corresponding to the partition $\{ z_0 , \dots , z_k \}$.

\begin{example}
\em
In general,
there is no hope of proving that the lower and upper measures of definable subsets of $[0,1]^n$ coincide if we replace the definition of $\sim$ on $\ma^{\geq 0}$ by $x\sim y$ iff $v(x)=v(y)$.  To see this, consider the function $f\colon [\epsilon^2 ,\epsilon ] \to [0,1]$ given by $f(x)=\frac{\epsilon^2 }{x}$, where $\epsilon \in \ma^{>0}$.  Let $\delta \in \ma^{>0}$ be such that $$v(\epsilon^p )< v(\delta )< v(\epsilon^2 ),$$ where $p\in \mathbb{Q}^{<2}$.  It is easy to see that then $\underline{\mu}(0,f)=\widetilde{\epsilon}^2$, but there is no finite partition of $[0,1]$ so that the corresponding upper sum $U$ of $f$ would be such that $U \leq \widetilde{\delta}$.
\end{example}

\smallskip\noindent
Until Theorem \ref{main} has been proven, we shall write $\underline{\mu}C$ and $\overline{\mu}C$ for the lower and upper measures of a cell $C \subseteq [0,1]^n$ computed as in part 1 of Definition \ref{defmeasure} (this is in contrast to $\underline{\mu}_{\mathcal{D}}C$ and $\overline{\mu}_{\mathcal{D}}C$ which are computed as in part 2.).


\begin{lemma}\label{smalllemma}
Let $X \subseteq [0,1]^n$ be definable with $\interior{ \pi X }=\emptyset$, and let $\mathcal{D}$ be a decomposition of $R^n$ into cells that partitions $X$.  Then there is no $x\in \bigcup \underline{\mu }_{\mathcal{D}} X$ with $x>\ma$.
\end{lemma}
\begin{proof}
The proof is by induction on $n$.  The case $n=1$ is clear, so suppose the lemma holds for $1,\dots ,n$, and let $X\subseteq [0,1]^{n+1}$.  Suppose $X=D_1 \cup \dots \cup D_m$, where $D_i \in \mathcal{D}$.  Assume 
towards a contradiction that $x\in \bigcup \underline{\mu}X$ is so that $x>\ma$.  Then there is $i\in \{ 1,\dots ,m \}$ such that $\bigcup \underline{\mu}D_i$ contains some $x>\ma$.  Then $\interior{D_i } \not=\emptyset$, so suppose $D_i =(f,g)$ and set $h=g-f$.
There are $$0=y_0 < y_1 < \dots < y_k = 1$$ so that $ \bigcup \sum_{i=0}^{k-1} \widetilde{y_i } \cdot \underline{\mu } h^{-1}[y_i ,y_{i+1}]$ contains an element 
$>\ma$,
hence $$\widetilde{y_i }\cdot \underline{\mu }h^{-1}[y_i ,y_{i+1}] = \widetilde{a}$$ for some $a\in V^{> \ma}$ and $i\in \{0,\dots ,k-1   \}$.  It follows that $y_i >\ma$, and there is $x\in \bigcup \underline{\mu} h^{-1}[y_i ,1]$ with $x>\ma$.  But then, by the inductive assumption, $\interior{\pi h^{-1}[y_i ,y_{i+1}]} \not=\emptyset$, hence $$\interior{ \pi \big( h^{-1} [y_i ,y_{i+1}] \times [0,y_i ] \big) } \not=\emptyset ,$$ a contradiction.
\end{proof}
\begin{lemma}\label{inner}
If $X=(f,g) \subseteq [0,1]^n$ is an open cell with $\interior{\pi X}=\emptyset$, then for each $a\in V^{\geq 0}$ with $\widetilde{a}<\underline{\mu}X$ there is $y\in [0,1]$ so that $$\widetilde{a}<\widetilde{y}\cdot \underline{\mu}h^{-1}[y,1],$$ where $h=g-f$.
\end{lemma}
\begin{proof}
Immediate from Lemma \ref{smalllemma} and iv) in the second part of Remark \ref{rem}.
\end{proof}

\begin{theorem}\label{main}
Let $X \subseteq [0,1]^{n}$ be definable.  Then $$\underline{\mu}_{\mathcal{E}}X=\underline{\mu}_{\mathcal{F}}X= \overline{\mu}_{\mathcal{F}}X = \overline{\mu}_{\mathcal{E}}X,$$
for all decompositions $\mathcal{E}$ and $\mathcal{F}$ of $R^n$ into cells that partition $X$.  

\smallskip\noindent
We shall refer to the common value of the upper and lower measures of $X$ as {\em the measure of $X$\/} and denote it by $\mu X$.
\end{theorem}

\begin{proof}  We may as well assume $\interior{X}\not=\emptyset$.  The proof is by induction on $n$.  The case when $n=1$ holds by Lemma \ref{addcong}, so assume inductively that the theorem holds for $1, \dots ,n$, and let $X \subseteq [0,1]^{n+1}$.

\begin{trivlist}

\item[{\bf Case 1.\/}] Suppose $\interior{\pi X} = \emptyset$.
\begin{trivlist}
\item[{\em Claim 1.\/}]
Let $X=(f,g)$ be an open cell.  Then $\underline{\mu}X=\overline{\mu}X$.
\item[{\em Proof of Claim 1.\/}]
We set $h=g-f$, and we
define $$A := \sup_{y \in [0,1]} \{ \widetilde{y} \cdot \mu (h^{-1}[y,1])  \} \in \widetilde{V} ,$$ where the expression $\mu h^{-1}[y,1]$ makes sense by the inductive assumption.
We shall say that {\em property $\ast$ holds for $h$\/} if there is $x\in \ma^{>0}$ such that $$\widetilde{y} \cdot \mu (h^{-1} [y,1]) <\widetilde{x}$$ for all $y\in [0,1]$, and 
there is $y \in [0,1]$ and $q \in \mathbb{Q}^{>1}$ so that $$\widetilde{x}^q <\widetilde{y} \cdot \mu (h^{-1} [y,1]).$$
We distinguish two cases.

\begin{enumerate}
\item 
First, assume that property $\ast$ holds for $h$.

Let $x\in \ma^{>0}$ witness that $\ast$ holds for $h$.  We set $$\mathcal{S}:=\{ q \in \mathbb{Q}^{>1}\colon \, \exists y\in [0,1] \;\, \widetilde{x}^{q}< \widetilde{y} \cdot \mu h^{-1}[y,1] \}.$$  Then $\mathcal{S}$ is a nonempty subset of $\mathbb{R}$ that is bounded below, hence the infimum of $\mathcal{S}$ exists in $\mathbb{R}$.  We set $c:= \inf \mathcal{S}$.  
\begin{subclaim}
Let $q_1 , q_2 \in \mathbb{Q}^{>0}$ be so that $q_1 <c < q_2$.  Then 
$$\widetilde{x}^{q_2 }<\underline{\mu}(0,h)\leq \overline{\mu}(0,h)<\widetilde{x}^{q_1}.$$
\end{subclaim}
\begin{proof of subclaim}
We first show that $\widetilde{x}^{q_2} < \underline{\mu}(0, h)$.  By the definition of $c$, we can find $q \in \mathcal{S}$ so that $c<q <q_2 $, and we let $y \in [0,1]$ satisfy $$\widetilde{x}^{q} < \widetilde{y}\cdot \mu h^{-1} [y,1] .$$  Then $$\widetilde{x}^{q_2 }< \widetilde{y} \cdot \mu h^{-1} [y,1] \leq \underline{\mu}(0,h).$$

To prove $\overline{\mu}(0,h)<x^{q_1}$,
let $q_3 \in \mathbb{Q}^{>0}$ and a positive integer $l$ be such that $q_1 + 2q_3 < c$ and $q_1 + q_3 < l q_3$.
Then the upper sum of $h$ corresponding to the partition $\{ 0,x^{lq_3}, x^{(l-1)q_3}\dots , x^{q_3}, 1 \}$ of $[0,1]$ is 
$$ \mu h^{-1}[x^{q_3},1] +  \sum_{i=1}^{l-1}\widetilde{ x}^{i q_3} \mu h^{-1}[x^{(i+1)q_3},x^{iq_3}] + \widetilde{x}^{lq_3} \mu h^{-1}[0,x^{lq_3}].$$  
Now
$\mu h^{-1}[x^{q_3 } ,1]  <\widetilde{x}^{q_1 + q_3}$, because else
$\mu h^{-1}[x^{q_3 } ,1]  \geq \widetilde{x}^{q_1 + q_3}$ would imply $\widetilde{x}^{q_3} \cdot \mu h^{-1}[x^{q_3 } ,1]  \geq \widetilde{x}^{q_1 + 2q_3}$, a contradiction with $\widetilde{x}^{c}<\widetilde{x}^{q_1 + 2q_3}$.

For $i=1,\dots ,l-1$, we have
$$\widetilde{x}^{iq_3} \mu h^{-1} [x^{(i+1)q_3},x^{iq_3}]<
\widetilde{x}^{q_1 + q_3},$$ because else
$$\widetilde{x}^{iq_3} \mu h^{-1} [x^{(i+1)q_3},x^{iq_3}]\geq
\widetilde{x}^{q_1 + q_3}$$ would imply
$$\widetilde{x}^{(i+1)q_3} \mu h^{-1} [x^{(i+1)q_3},x^{iq_3}]\geq
\widetilde{x}^{q_1 + 2q_3},$$ again a contradiction with $\widetilde{x}^{c}<\widetilde{x}^{q_1 + 2q_3}$.

Also, $$\widetilde{x}^{lq_3} \mu h^{-1} [0, x^{l q_3}]  \leq \widetilde{x}^{lq_3} < \widetilde{x}^{q_1 + q_3}.$$
So the upper sum of $h$ corresponding to $\{ 0,x^{lq_3}, x^{(l-1)q_3}\dots , x^{q_3}, 1 \}$ is smaller than $(l+1)\cdot \widetilde{x}^{q_1 + q_3}=\widetilde{x}^{q_1 + q_3}<\widetilde{x}^{q_1}$.
\end{proof of subclaim}

It now follows that $\underline{\mu}(0,h)=\overline{\mu}(0,h)$: If not, then we can find $y,z \in V^{>0}$ so that $x^{q_2 }<y<z<x^{q_1}$ for all $q_1 , q_2 \in \mathbb{Q}^{>0}$ with $q_1 <c<q_2$, and $y\not\sim z$.  Hence $y<z^q$ for some $q\in \mathbb{Q}^{>1}$. Then 
$$x^{q_2} < y<z^{q}<x^{qq_1}$$ for all $q_1 , q_2 \in \mathbb{Q}^{>0}$ with $q_1 < c < q_2$.  But picking $q_1$ so that $q q_1 > c$ yields a contradiction with $\widetilde{x}^{q_2}<y$ for all $q_2 \in \mathbb{Q}^{>c}$.



\item Suppose $\ast$ does not hold for $h$.

In this case, if $x\in \ma^{>0}$, then either $A <\widetilde{x}^p$ for all $p\in \mathbb{Q}^{>0}$, or $\widetilde{x}^p<A$, for all $p\in \mathbb{Q}^{>0}$.
We shall show that $\overline{\mu}(0,h) \leq A \leq \underline{\mu} (0,h)$.

To prove that $A \leq \underline{\mu} (0,h)$, let $a\in V^{>0}$ be such that $\widetilde{a} < A$.  Then
we can find $y\in [0,1]$ so that $\widetilde{a}<\widetilde{y}\cdot \mu h^{-1} [y,1] \leq \underline{\mu}(0,h)$.

\smallskip\noindent
To see that $\overline{\mu} (0,h) \leq A$, let $y \in V^{>0}$ be such that $A<\widetilde{y}$.

First, suppose $\ma < y <1$.  Then $\mu h^{-1} [\frac{y}{2},1]<\widetilde{ \big( \frac{y}{2}  \big) }$, because else $$\widetilde{\big(  \frac{y}{2} \big) }\cdot \overline{\mu}h^{-1}[\frac{y}{2},1]\geq \widetilde{\big(  \frac{y}{2}     \big)}^2  > A,$$ would yield a contradiction with the definition of $A$.  So $$\overline{\mu}(0,h) \leq \mu h^{-1}[\frac{y}{2},1] + \widetilde{\frac{y}{2}} \cdot \overline{\mu }h^{-1}[0,\frac{y}{2}]<\widetilde{\big( \frac{y}{2} \big) }+\widetilde{\big( \frac{y}{2}\big) }=\widetilde{y}.$$

So assume that $y\in  \ma^{>0}$.  Then $A<\widetilde{y}^2$, because $\ast$ fails for $h$.  Hence $\mu h^{-1} [y,1] <\widetilde{y}$, else $\widetilde{y}\cdot \mu h^{-1} [y,1]\geq \widetilde{y}^2 >A$, a contradiction.  So $$\overline{\mu}(0,h) \leq \mu h^{-1} [y,1] + \widetilde{y} \cdot \mu h^{-1} [0,y] < \widetilde{y}+\widetilde{y}=\widetilde{y}.$$

It follows that $\underline{\mu}(0,h)=\overline{\mu}(0,h) = \mu (0,h)$.

This finishes the proof of Claim 1.

\end{enumerate}

\item[{\em Claim 2.\/}] Let $X=(f,g)$ be an open cell, and let $\mathcal{D}$ be a decomposition of $R^{n+1}$ into cells that partitions $X$.  Then $\mu X = \mu_{\mathcal{D}}X$.  
\item[{\em Proof of Claim 2.\/}]
Let $D_1 , \dots ,D_k \in \mathcal{D}$ be open with $X=_0 D_1 \cup \dots \cup D_k$. 
To see that $\mu X\leq \sum_{i=1}^{k}\mu D_i$, let $a\in V^{\geq 0}$ be so that $\widetilde{a}<\mu X$.  By Lemma \ref{smalllemma}, $a\in \ma^{\geq 0}$.  We need to show that $\widetilde{a}<\sum_{i=1}^{k}\mu D_i$.  By Lemma \ref{inner}, we can find $y\in [0,1]$ such that $\widetilde{a}< \widetilde{y}\cdot \mu h^{-1}[y,1]$, where $h=g-f$.  
\begin{itemize}
\item
If there is no $x\in \bigcup \mu h^{-1}[y,1]$ with $x>\ma$, then, using the inductive assumption, we can find $D \in \{ D_1 , \dots ,D_k  \}$ so that $$\mu h^{-1}[y,1] = \mu (h^{-1}[y,1] \cap p^{n+1}_{n}D).$$
Let $\{ E_1 , \dots ,E_m  \}$ be the subset of $\{ D_1 , \dots , D_k \}$ consisting of all $D_i$'s with $p^{n+1}_{n}D_i=p^{n+1}_{n}D$.  For each $i \in \{ 1,\dots ,m \}$, let $E_i =(f_i , g_i )$, set $h_i = g_i - f_i$, and define $$F_i := \{ x\in h^{-1}[y,1]\cap p^{n+1}_{n}D:\; h_i  (x) \geq h_j (x) \mbox{ for }j=1,\dots,m \}.$$ 
Then $h^{-1}[y,1]\cap p^{n+1}_{n}D = \bigcup_{i=1}^{m}F_{i}$, and hence we can take $j \in \{ 1,\dots ,m \}$ so that $$\mu F_j =\mu (h^{-1}[y,1] \cap p^{n+1}_{n}D).$$
We claim that $\widetilde{a}< \mu E_j$.  This is because if $y\in \ma^{>0}$, then $\widetilde{y}\leq \widetilde{h_j (x)}$ for each $x\in F_j$.  And if $y>\ma$, then $\widetilde{y}\cdot \mu h^{-1}[y,1] = \mu h^{-1}[y,1]$ and $(g_j - f_j )(x)>\ma$ for each $x\in F_j$.
\item
Now suppose there is $x\in \bigcup \mu h^{-1}[y,1]$ with $x>\ma$.  Let $D \in \{ D_1,\dots ,D_k \}$ be such that $\bigcup \mu (h^{-1}[y,1] \cap p^{n+1}_{n} D)$ contains some $x>\ma$.  Then $$\widetilde{y}\cdot \mu h^{-1}[y,1] = \widetilde{y} \cdot \mu (h^{-1}[y,1] \cap p^{n+1}_{n}D)= \widetilde{y}.$$  Define $\{ E_1 , \dots ,E_m \}$ and the sets $F_i$ for $D$ as in the previous case.  Then for some $i\in \{ 1,\dots ,m \}$, there is $x \in \bigcup \mu F_i$ so that $x>\ma$.  Hence $\mu E_i >\widetilde{a}$.
\end{itemize}
To see that $\sum_{i=1}^{k}\mu D_i \leq \mu X$, let $a\in V^{\geq 0}$ be such that $\widetilde{a}<\sum_{i=1}^{k}D_i$.  By Lemma \ref{smalllemma}, $a\in \ma^{\geq 0}$.  Then $\sum_{i=1}^{k}\mu D_i = \mu D_j$ for some $j \in \{1,\dots ,k  \}$.  Let $D_j = (f_j , g_j )$ and set $h_j = g_j -f_j$.  Then there is $y\in [0,1]$ with $$\widetilde{y}\cdot \mu (h_{j}^{-1}[y,1]) > \widetilde{a},$$ and $$\widetilde{y}\cdot \mu \big( h_{j}^{-1} [y,1]\big) \leq \widetilde{y}\cdot h^{-1}[y,1] \leq \mu X.$$
This finishes the proof of Claim 2.

\item[{\em Claim 3.\/}] Let $X$ be a definable set, and let $\mathcal{C}$ and $\mathcal{D}$ be decompositions of $R^{n+1}$ into cells that partition $X$.  Then $\mu_{\mathcal{C}}X=\mu_{\mathcal{D}}X$.
\item[{\em Proof of Claim 3.\/}]
Let $\mathcal{E}$ be a decomposition of $R^{n+1}$ into cells which is a common refinement of $\mathcal{C}$ and $\mathcal{D}$.  Then

$$\mu_{\mathcal{D}}X=\sum_{ D_i \subseteq X} \mu D_i =\sum_{D_i \subseteq X} \sum_{E_{ij} \subseteq D_i } \mu E_{ij} = \sum_{C_{k}\subseteq X} \sum_{E_{kl}\subseteq C_{k}} \mu E_{kl} = \sum_{C_{k}\subseteq X} \mu C_{k} = \mu_{\mathcal{C}}X,$$ where $D_i \in \mathcal{D}$, $E_{ij}, E_{kl}\in \mathcal{E}$ and $C_{k}\in \mathcal{C}$.

This finishes the proof of Claim 3, and we have thus proven Case 1.
\end{trivlist}
\end{trivlist}

\item[{\bf Case 2.\/}] $\interior{\pi X} \not=\emptyset$.  

\smallskip\noindent
Since $\pi X$ is definable in the o-minimal structure $\mathbb{R}_0$, it is Lebesgue measurable, and $\underline{\mu}_{\mathcal{P}}\pi X =\overline{\mu}_{\mathcal{P}}\pi X = \widetilde{a}$, where $a\in \mathbb{R}^{>0}$ is the Lebesgue measure of $\pi X$, and $\mathcal{P}$ is any decomposition of $\mathbb{R}^{n+1}$ into cells that partitions $\pi X$.  We shall thus write $\mu Y$ instead of $\overline{\mu}_{\mathcal{P}} Y$ and $\underline{\mu}_{\mathcal{P}}Y$ if $Y$ is an $\mathbb{R}_0$-definable subset of $[0,1]^{m} \subseteq \mathbb{R}^m$.

Our aim is to show that
$\underline{\mu}_{\mathcal{D}} X = \overline{\mu}_{\mathcal{D}} X = \widetilde{a}$.
Since this is clearly satisfied when $X\subseteq [0,1]$, we may assume that the inductive assumption holds in this a priori stronger form.  
\begin{trivlist}
\item[\em Claim 1.\/] Suppose $X=(f,g) \subseteq [0,1]^{n+1}$ is a cell.
Then 
$\underline{\mu} X = \overline{\mu} X = \widetilde{a}$.
\item[\em Proof of Claim 1.\/] We set $h=g-f$.  
By o-minimality of $\mathbb{R}_0$, there are $\mathbb{R}_0$-definable functions $f_0$, $g_0$, and $h_0$ with $$\domain{f_0}=\domain{g_0}=\domain{h_0}=_0 \pi p^{n+1}_{n}X$$ and such that for all $x\in \domain{f_0 }$, $$f_0(x) = \pi f(x'), \; g_0 (x)=\pi g(x') \mbox{ and } h_0 (x)=\pi h(x'),$$ where $x'\in p^{n+1}_{n}X$ is such that $\pi (x')=x$.


Let $\mathcal{C}_0$ be a decomposition of $\mathbb{R}^{n}$ into cells that partitions the domain of $h_0$ and is such that whenever $C \in \mathcal{C}_0$ is open and $C \subseteq \domain{h_0}$, then $h_0$ is differentiable on $C$ and each $\frac{\partial h_0 }{\partial x_i}$ has constant sign. 

By Fact \ref{closed}, we can find for each $C\in \mathcal{C}_0$ an $R$-definable set $X_C$ so that $\pi X_C = \closure{C}$.
Let $\mathcal{D}_0$ be a decomposition of $R^{n}$ partitioning $p^{n+1}_{n}X$ and $X_C$ for each $C\in \mathcal{C}_0$ with $C \subseteq \domain{h_0 }$.

\smallskip\noindent
\begin{subclaim}
Let $D\in \mathcal{D}_0$ be such that $D\subseteq p^{n+1}_{n}X$.  Set $X_D := (0,h) \cap (D\times R)$ and suppose $\interior{\pi X_D} \not=\emptyset$.  Then $\underline{\mu} X_D = \overline{\mu} X_D=\widetilde{d}$, where $d$ is the Lebesgue measure of $\pi X_D$.
\end{subclaim}
\begin{proof of subclaim}
We replace for the moment $h$ with $h|_{D}$, and $h_0$ with $h_0 |_{\interior{\pi D}}$.
We shall show $\overline{\mu}(0,h) \leq \widetilde{d}$ and $\widetilde{d} \leq \underline{\mu}(0,h)$. To prove the first inequality, let $d'\in \mathbb{R}$ be such that $\widetilde{d}<\widetilde{d'}$.
We wish to show that $ \overline{\mu}(0, h) < \widetilde{d'}$.  Let $0=a_0 < \dots < a_k =1$ be real numbers so that
$$\sum_{i=0}^{k-1} \widetilde{a_{i+1}} \cdot \mu h_{0}^{-1}[a_{i} , a_{i+1}] < \widetilde{d'}.$$
By Fact \ref{intersection}, for each $i$, we can find $\epsilon_i \in \ma^{\geq 0}$ so that

$$\pi \big( \Gamma h \cap (R^n \times [b_{i} -\epsilon_{i} , b_{i+1} +\epsilon_i ])    \big) = \Gamma h_0 \cap (\mathbb{R}^n \times [a_{i} , a_{i+1}])$$ up to a set of dimension $<n$, where $b_{i} , b_{i+1}\in R$ are such that $\pi b_{i} =a_{i}$ and $\pi b_{i+1}=a_{i+1}$.
Inductively, $$\mu h^{-1}[b_{i} -\epsilon_i , b_{i+1} + \epsilon_i ] = \mu \pi h^{-1}[b_{i} -\epsilon_i , b_{i+1} + \epsilon_i ],$$ hence $$\mu h^{-1}[b_{i} -\epsilon_i , b_{i+1} + \epsilon_i ] = \mu h_{0}^{-1}[a_{i} , a_{i+1}].$$
So
\[
\begin{array}{lll}
&&\sum_{i=0}^{k-1} \widetilde{b}_{i+1} \cdot \mu  h^{-1}[b_i , b_{i+1} ] \leq \\  && \sum_{i=0}^{k-1} \widetilde{(b_{i+1} + \epsilon_i )} \cdot \mu h^{-1}[b_{i}-\epsilon_i , b_{i+1} + \epsilon_i ] 
 = \\
&& \sum_{i=0}^{k-1} \widetilde{a_i} \cdot \mu  h_{0}^{-1} [a_i , a_{i+1}] < \widetilde{d'}.

\end{array}
\]


Next, we need to show that $\widetilde{d}  \leq \underline{\mu } (0,h)$. 
There are two cases to be considered.

\begin{itemize}
 \item[1.] Suppose $\frac{\partial h_0 }{\partial x_j} =0$ for all $j$.  

Then $\widetilde{d} = \mu p^{n+1}_{n}(0,h_0 ) \cdot \widetilde{ h_0 (x)}$ for any $x\in p^{n+1}_{n}(0,h_0 )$.  Let $b\in V^{>\ma}$ be such that $\pi (b)=h_0 (x)$.  By Fact \ref{intersection}, we can find $\epsilon \in \ma^{\geq 0}$ so that $$\pi (\Gamma h \cap (R^n \times [b-\epsilon , b+ \epsilon ])) = \Gamma h_0 $$ up to a set of dimension $<n$.  Then
$$\widetilde{d} = \widetilde{h_0 (x)}\cdot \mu p^{n+1}_{n}(0,h_0 )  \leq  \widetilde{(b-\epsilon )}\cdot \mu h^{-1} [b-\epsilon , b+ \epsilon ] + \widetilde{(b+\epsilon )} \cdot \mu h^{-1}[b+\epsilon , 1]$$ by the inductive assumption.

\item[2.] Suppose $\frac{\partial h_0}{\partial x_j} \not= 0$ for some $j$.  Let $d'\in V^{>\ma}$ be such that $\widetilde{d'} <\widetilde{d}$.  We wish to show that $\widetilde{d'}<\underline{\mu } (0,h)$.

Let $$0=a_0 <\dots <a_k =1$$ be elements of $\mathbb{R}$ so that $\widetilde{d'} < \sum_{i=0}^{k} \widetilde{a_i } \cdot \mu h_{0}^{-1}[a_i , a_{i+1}]$, and let $$0=b_0 < \dots < b_k = 1$$ be elements of $R$ such that $\pi b_i = a_i$ for each $i$.
Then, for each $i$, $\mu \pi h^{-1} [b_i , b_{i+1}] = \mu h_{0}^{-1} [a_i , a_{i+1}]$:  The inequality $$\mu \pi h^{-1} [b_i , b_{i+1}] \leq \mu h_{0}^{-1} [a_i , a_{i+1}]$$ is clear by the inductive assumption.  To prove the other inequality, let $\epsilon \in \ma^{>0}$ be such that 
$$\pi h^{-1}[b_i - \epsilon , b_{i+1} + \epsilon ]=h^{-1}_{0}[a_{i},a_{i+1}].$$
Then $$\pi h^{-1}[b_i -\epsilon , b_{i+1} + \epsilon ]= \pi h^{-1}[b_i -\epsilon , b_i ]\cup \pi h^{-1}[b_i , b_{i+1}] \cup \pi h^{-1} [b_{i+1}, b_{i+1}+\epsilon ],$$ where the sets on the right-hand side are disjoint apart from a set of dimension $<n$.  Hence
$$\mu h^{-1}_0 [a_i , a_{i+1}] = \mu \pi h^{-1}[b_{i}-\epsilon , b_i ]+\mu \pi h^{-1}[b_i , b_{i+1}]+\mu \pi h^{-1}[b_{i+1} , b_{i+1} + \epsilon ].$$
But $$\mu \pi h^{-1}[b_i -\epsilon ,b_i ]=\mu \pi h^{-1}[b_{i+1},b_{i+1}+\epsilon ]=0,$$
because $\mu h^{-1}_{0}(a_i ) = \mu h^{-1}_{0}(a_{i+1}) =0$.


It follows that $$\widetilde{d'}< \sum_{i=0}^{k-1} \widetilde{a_i } \cdot \mu h_{0}^{-1} [a_i , a_{i+1}] = \sum_{i=0}^{k-1} \widetilde{b_i} \cdot \mu  h^{-1}[b_i , b_{i+1}].$$ 
\end{itemize}
This proves $\widetilde{d} < \underline{\mu}(0,h)$, and hence
$\underline{\mu}X_D =\overline{\mu}X_D =\widetilde{d}$. 
\end{proof of subclaim}

\smallskip\noindent
Now let $p^{n+1}_{n}X=\bigcup_{i=1}^{k}D_i ,$ where each $D_i \in \mathcal{D}_0$.  Then each set $X\cap (D_i \times R)$ is a cell, and
$$X= \big( (D_1 \times R)  \cap X \big) \dot\cup  \dots \dot\cup \big( (D_k \times R)\cap X \big).$$
Let $\mathcal{D}$ be a decomposition of $R^{n+1}$ into cells such that $ (D_i \times R) \cap X \in \mathcal{D}$ for each $i\in \{1,\dots ,k  \}$, and let $I\subseteq \{1,\dots ,k \}$ consist of all the $i$ with
$$\interior{\pi ( (D_i \times R) \cap X)} \not= \emptyset .$$
By the subclaim, if $i\in I$, then
we can find $a_i \in V^{>\ma}$ so that $\pi a_i$ is the Lebesgue measure of $\pi \big( (D_i \times R) \cap X \big)$ and  $$\mu \big((D_i \times R) \cap X \big) = \underline{\mu } \big( (D_i \times R)\cap X\big)= \overline{\mu} \big( (D_i \times R) \cap X \big)= \widetilde{a_i }.$$
For $i\in \{ 1,\dots ,k \}\setminus I$, we set $a_i = 0$.
Note that $\sum_{i=1}^{k} \pi a_i =\pi a$.
To prove $\underline{\mu}X=\overline{\mu}X=\widetilde{a}$, let $a'\in R^{>\ma}$ be such that $\widetilde{a}<\widetilde{a'}$.
We need to show $\overline{\mu}X < \widetilde{a'}$.  Let for each $i\in \{ 1,\dots ,k  \}$, $$0=b_{i_0} < b_{i_1}<\dots  <b_{i_{k_i}}=1$$ be a partition of $[0,1]$ so that the corresponding upper sum of $h|_{D_i }$ has measure at most $\widetilde{a_i } + \widetilde{\frac{a'-a}{k}}$.  
Such a partition exists for $i\in I$ by the subclaim, and for $i\in \{ 1,\dots ,k  \}\setminus I$ by Case 1.

Now let $\{ b_0 , \dots ,b_m \}$ be a partition of $[0,1]$ which is a common refinement of all $\{ b_{i_0},\dots ,b_{i_{k_i}}  \}$ where $i=1,\dots ,k$.  Then the upper sum of $h|_{D_i}$ corresponding to this new partition is again at most 
$\widetilde{a_i } + \widetilde{\frac{a'-a}{k}}$.  Furthermore, $$\sum_{i=1}^{m} \widetilde{b_{i}} \cdot \mu h^{-1}[b_{i-1},b_i ] =\sum_{i=1}^{m} (\sum_{j= 1}^{k} \widetilde{b_i } \cdot \mu (h^{-1}[b_{i-1},b_i ] \cap D_j )) < \widetilde{a'},$$ where the first equality follows from the inductive assumption.
The inequality $\widetilde{a} \leq \underline{\mu}X$ is proved similarly.
This finishes the proof of Claim 1.
\end{trivlist}

\begin{trivlist}
\item[{\em Claim 2.\/}]
 $\underline{\mu}X_{\mathcal{E}}=\overline{\mu}_{\mathcal{E}}X=\widetilde{a}$.
\item[{\em Proof of Claim 2.\/}]
Let $E_1 , \dots ,E_k \in \mathcal{E}$ be open such that $X=_0 \bigcup_{i=1}^{k}E_i$.  Since $\interior{\pi E_i } \not=\emptyset$ for at least one $i$, we may as well assume (by Lemma \ref{smalllemma}) that $\interior{\pi E_i } \not=\emptyset$ for each $i$.  Now, by the above, $\underline{\mu}E_i = \overline{\mu}E_i = \widetilde{b}_i$, where $\pi b_i$ is the Lebesgue measure of $\pi D_i$.  Hence $$\underline{\mu}X=\sum_{i=1}^{k}\underline{\mu}D_i = \sum_{i=1}^{k}\overline{\mu}D_i = \overline{\mu}X.$$
\end{trivlist}
This finishes the proof of Claim 2, thus the proof of Case 2, and hence the proof of the theorem.
\end{proof}
\end{section}

\begin{section}{Measuring definable subsets of $V^n$ and invariance of $\mu$ under isomorphisms}
The following definition is from \cite{nip}.
By $J\phi (x)$ we denote the determinant of the Jacobian of a diffeomorphism $\phi$ at $x$. 
\begin{definition}\label{iso}
Let $SB[n]$ be the lattice of all $R$-definable subsets of $V^n$, and let $X,Y \in SB[n]$.  An {\em isomorphism\/} $\phi \colon X\to Y$ is defined to be a definable $C^1$-diffeomorphism $\phi \colon U \to V$, where $U$ and $V$ are open definable subsets of $R^n$, $X\subseteq_0 U$, $Y \subseteq_0 V$, $|J\phi (x)|=1$ for all $x\in U\cap X$ up to a set of dimension $<n$, and $\phi (X)=_0 Y$.
\end{definition}
Let $C\subseteq V^n$ be an open cell with $C=(f_n ,g_n )$ and $p^{n}_{k}C = (f_{k},g_{k})$ for $k=1,\dots ,n-1$.  Suppose that $f_i$ and $g_i$ are continuously differentiable for $i=2,\dots ,n$.  We define a map $$\tau_C = (\tau_1 , \dots ,\tau_n ) \colon C \to \tau C$$ by setting $\tau_k (x)=x_k - f_k (x_1 , \dots ,x_{k-1})$ for $x=(x_1 , \dots ,x_n ) \in C$ and $k=1, \dots ,n$.  It is routine to check that $\tau$ is an isomorphism $C\to \tau C$.
\begin{lemma}\label{box}
Let $X \subseteq [0,1]^n$ be definable and such that $\interior{\pi X}=\emptyset$.  Then for each $a\in V^{>0}$ with $\widetilde{a}<\mu X$, there is a cell $C \subseteq X$ and a box $B\subseteq_0 \tau_C C$ with $\mu B>\widetilde{a}$.
\end{lemma}
\begin{proof}
Let $a\in V^{\geq 0}$ be such that $\widetilde{a}<\mu X$.  
Let $\mathcal{D}$ be a decomposition of $R^n$ into cells that partitions $X$.  Suppose $X=D_1 \cup \dots \cup D_m$, where each $D_i \in \mathcal{D}$.  Since $\interior{\pi X}=\emptyset$, by Lemma \ref{smalllemma} we can find $D\in \{ D_1 , \dots ,D_m  \}$ so that $\mu D = \mu X$.  We shall find a box $B\subseteq_0 \tau_D (D)$ with $\mu D >\widetilde{a}$.

If $n=1$, then $\tau D$ is the required box. 
So assume the lemma holds for $1,\dots ,n$, and let $X\subseteq [0,1]^{n+1}$. Suppose $\tau_D D=(0,h)$.  Then we can find a partition $0=y_0 <y_1 <\dots <y_l =1$ of $[0,1]$ so that $$\widetilde{a}< \sum_{i=1}^{l}\widetilde{y}_{i-1} \cdot \mu h^{-1}[y_{i-1},y_i],$$ and $\sum_{i=1}^{l}\widetilde{y}_{i-1}\cdot \mu h^{-1}[y_{i-1},y_i ] = \widetilde{y}_{j-1} \cdot \mu h^{-1}[y_{j-1},y_j ]$ for some $j\in \{ 1,\dots ,l  \}$.  

If $\interior{\pi h^{-1}[y_{j-1},y_j ]}=\emptyset$, then $h^{-1}[y_{j-1},y_j ]$ contains a cell $C$ of measure $\mu h^{-1}[y_{j-1},y_j ]$, and $$\tau_C C\subseteq \tau_{p^{n+1}_{n}D} h^{-1}[y_{j-1},y_j]  \subseteq \tau_{p^{n+1}_{n}D} p^{n+1}_{n} D =  p^{n+1}_{n} \tau_D D .$$ Let $c\in V^{>0}$ be such that $\widetilde{c} < \mu h^{-1}[y_{j-1},y_j ]$ and $\widetilde{a}<\widetilde{y}_{j-1}\cdot \widetilde{c}$.
By the inductive assumption, $\tau_C C$ contains a box $B_0$ with $\widetilde{c}<\mu B_0$.  Then $B_0 \times [0,y_{j-1}] \subseteq \tau_{D} D$ and $\mu (B_0 \times [0,y_{j-1}]) >\widetilde{a}$.

If $\interior{\pi h^{-1}[y_{j-1},y_j ]}\not=\emptyset$, then $h^{-1}[y_{j-1},y_j ]$ contains a cell $C$ such that $\interior{\pi C}\not=\emptyset$.  Then $\interior{\pi \tau_C C} \not=\emptyset$ and, by Fact \ref{vbox}, $\tau_C C$ contains a box $B_0$ of measure $>\ma$.  Then $B_0 \times [0,y_{j-1}]$ is as required.
\end{proof}

\begin{lemma}\label{boxlemma}
Let $X \subseteq [0,1]^n$ be definable with non-empty interior, and let $a\in V^{>0}$ be such that $\mu X <\widetilde{a}$.  Then there are open cells $C_1, \dots ,C_k \subseteq [0,1]^n$ so that $X=_0 C_1 \dot\cup \dots \dot\cup C_k$, and for each $i\in \{ 1,\dots ,k \}$ there are 
boxes $B_{i1} , \dots ,B_{ik_i } \subseteq [0,1]^n$ with $\tau C_i \subseteq \bigcup_{j=1}^{k_i} B_{ij}$ and $\sum_{i=1}^{k}\sum_{j=1}^{k_i
}\mu B_{ij} < \widetilde{a}$.
\end{lemma}
\begin{proof} 
First assume that $X=(f,g)$ is a cell and set $h=g-f$.
The proof is by induction on $n$.  If $n=1$, then $X=(c,d)$ for some $c,d \in V^{\geq 0}$.  Then $\tau X_{(f,g)} \subseteq [0,d-c]$ and $\mu [0,d-c]=\mu X<\widetilde{a}$.

So suppose the lemma holds for $1,\dots ,n$, and let $X \subseteq [0,1]^{n+1}$. Let $$0=y_0 < y_1 <\dots <y_k =1$$ be such that $\sum_{i=1}^{k}\widetilde{y_i } \cdot \mu h^{-1} [y_{i-1},y_i ]<\widetilde{a}$.
\begin{trivlist}
\item[{\bf Case 1.\/}] There is no $c\in \ma^{\geq 0}$ with $\mu X<\widetilde{c}$.

\smallskip\noindent
 In this case $a>\ma$, and we fix $b\in V^{>\ma}$ so that $\mu X<\widetilde{b}<\widetilde{a}$.  It suffices to prove the conclusion of the lemma for each $$X':=X\cap (h^{-1}[y_{i-1},y_i] \times [0,1])$$ instead of $X$ and $\widetilde{c}:=\widetilde{b}+\widetilde{\frac{a-b}{k}}$ in place of $\widetilde{a}$.  
\begin{enumerate}
\item 
If $y_i \in \ma^{\geq 0}$, then let $\mathcal{D}$ be any decomposition of $R^{n+1}$ into cells partitioning $X'$.  For each open $D\in \mathcal{D}$ with $D\subseteq X'$ we have $\tau_{D}D \subseteq [0,1]^n \times [0,y_i ]$, which is a box of measure $\widetilde{y_i }$, and $l \cdot \widetilde{y_i} < \widetilde{c}$ for any non-negative integer $l$.

If $\mu h^{-1}[y_{i-1},y_i ]<\widetilde{d}$ for some $d\in \ma^{\geq 0}$, then we use the inductive assumption to find open cells $C_1 , \dots ,C_k$ so that $$\mu h^{-1}[y_{i-1},y_i ] =_0 C_1 \cup \dots \cup C_k ,$$ and for each $i \in \{  1,\dots ,k \}$ a family of boxes $\{ B_{ij} \colon j=1,\dots ,k_i \}$ covering $\tau_{C_i} C_i$ so that $\sum_{i=1}^{k_i} \mu B_{ij}<\widetilde{d}$.  Then the cells $X' \cap (C_i \times [0,1])$ and the families of boxes $$\{ B_{ij}\times [0,y_i ]\colon \; j=1,\dots ,k_i \},$$ where $i=1,\dots ,k$, are as in the conclusion of the lemma.

\item 
So suppose $y_i > \ma$, and there is no $d\in \ma^{\geq 0}$ with $\mu h^{-1}[y_{i-1},y_i ] < \widetilde{d}$.  Then $\mu h^{-1}[y_{i-1},y_i ] <\widetilde{\frac{c}{y_i }}$.   By the inductive assumption, we can find open cells $C_1 , \dots , C_k \subseteq [0,1]^n$ so that $$h^{-1}[y_{i-1},y_i ] =_0 C_1 \cup \dots \cup C_k ,$$ and for each $i \in \{1,\dots ,k \}$ a family of boxes $\{ B_{ij}\colon \; j=1,\dots ,k_i \}$ covering $\tau_{C_i} C_i$ with $$\sum_{i=1}^{k}\sum_{j=1}^{k_i} \mu B_{ij}<\widetilde{\frac{c}{y_i }}.$$  Then the cells $X' \cap (C_i \times [0,1])$ and the families of boxes $$\{ B_{ij} \times [0,y_i ] \colon j=1,\dots ,k_i \},$$ where $i=1,\dots ,k$, are as required.

\end{enumerate}

\item[{\bf Case 2.\/}] There is $c\in \ma^{>0}$ with $\mu X <\widetilde{c}$.

\smallskip\noindent
In this case we may assume that $a\in \ma^{>0}$. We fix $b\in \ma^{>0}$ with $$\mu X < \widetilde{b} < \widetilde{a}.$$  It suffices to prove the conclusion of the lemma for each set $$X':= X\cap (h^{-1}[y_{i-1},y_i ] \times [0,1])$$ in place of $X$.
\begin{enumerate}
\item Suppose $y_i \in \ma^{\geq 0}$ and $\mu h^{-1}[y_{i-1},y_i ]<\widetilde{c}$, where $c\in \ma^{\geq 0}$.

If $\widetilde{y_i}<\widetilde{a}$, then we find open cells $C_1 ,\dots ,C_k$ so that $$h^{-1}[y_{i-1},y_i ]=_0 C_1 \cup \dots \cup C_k ,$$ and for each $i$ a family of boxes $\{  B_{ij}\colon \; j=1,\dots ,k_i \}$ covering $\tau_{C_i} C_i$ such that $\sum_{i=1}^{k} \sum_{j=1}^{k_i} \mu B_{ij}< \widetilde{c}$.  Then the cells $X'\cap (C_i \times [0,1 ])$ and the families of boxes $\{ B_{ij}\times [0,y_i ] \colon j=1, \dots ,k_i \}$ are as required.


If $\widetilde{a}\leq \widetilde{y_i}$, then $\widetilde{b}<\widetilde{y_i }$, hence $z=\frac{b}{y_i } \in \ma^{>0}$.  Note that $\mu h^{-1}[y_{i-1},y_{i}]<\widetilde{z}$.
We proceed exactly as above, except that we require 
$$\sum_{i=1}^{k} \sum_{j=1}^{k_i} \mu B_{ij}< \widetilde{z}.$$

\item It is obvious how to handle the case when $y_i >\ma$.

\item
Suppose $y_i \in \ma^{\geq 0}$, and there is no $c\in \ma^{>0}$ so that $\mu h^{-1}[y_{i-1},y_i] <\widetilde{c}$.

Then $\widetilde{y_i } < \widetilde{a}$: If
$\widetilde{a}\leq \widetilde{y_i }$, then $\widetilde{b}<\widetilde{y_i}$, hence $\widetilde{b}<\widetilde{y_i }^{\frac{m}{n}}<\widetilde{y_i} \mbox{ for some } \frac{m}{n}\in \mathbb{Q}^{>1}.$  But $y_{i}^{\frac{m-n}{n}} \in \ma^{>0}$ and $\widetilde{y_i}\cdot \widetilde{y_i }^{\frac{m-n}{n}} = \widetilde{y_i}^{\frac{m}{n}}$, a contradiction with $\widetilde{y_i} \cdot \widetilde{\epsilon } < \widetilde{b}$ for all $\epsilon \in \ma^{>0}$.

It is now obvious how to handle this case as well.




\end{enumerate}

\end{trivlist}
We established the lemma for $X\subseteq [0,1]^n$ a cell.
Now suppose that $X\subseteq [0,1]^{n}$ is a definable set.  Let $\mathcal{D}$ be a decomposition of $R^n$ into cells partitioning $X$, and let $X=_0 D_1 \cup \dots \cup D_m$, where each $D_i \in \mathcal{D}$ is open.

The case when $a\in \ma^{\geq 0}$ follows immediately from Case 2 above.  So suppose there is no $c\in \ma^{>0}$ so that $\mu X <\widetilde{c}$.  Let $b\in V^{>0}$ be such that $\mu X<\widetilde{b}<\widetilde{a}$.  By Case 1, each $\tau D_i$ can be covered by finitely many boxes $B_{ij}$ of total measure $<\mu D_i + \widetilde{\frac{a-b}{m}}$.  Then the sum of the measures of all the boxes is $<\widetilde{a}$.
\end{proof}

\begin{theorem}\label{inv}
Let $X,Y \subseteq [0,1]^n$ be definable and isomorphic.  Then $\mu X = \mu Y$.
\end{theorem}

\begin{proof}
Let $\phi$ be an isomorphism $X\to Y$. 
It suffices to show that $\mu X \leq \mu Y$, since $\phi^{-1}$ is an isomorphism $Y\to X$.  If $\interior{\pi X}\not=\emptyset$, then the theorem is obvious from the proof of Theorem 6.5, p. 194 in \cite{real}.  

So suppose $\interior{\pi X}=\emptyset$.  Assume towards a contradiction that $\mu Y < \mu X$, and let $a\in \ma^{>0}$ be such that $\mu Y <\widetilde{a} <\mu X$.  By Lemma \ref{boxlemma}, we can find open cells $C_1 , \dots ,C_k \subseteq [0,1]^n$ so that $$Y=_0 C_1 \dot\cup C_2 \dot\cup \dots \dot\cup C_k ,$$ $\phi$ is defined on each $C_i$, and so that for each $i$, we can find a family of boxes $\{ B_{ij}: j=1,\dots ,k_i \}$ with $\tau_{C_i} C_i \subseteq_0 \bigcup_{j=1}^{k_i} B_{ij}$ and $\sum_{i=1}^{k} \sum_{j=1}^{k_i} \mu B_{ij} <\widetilde{a}$.
Then $$X=_0 \phi^{-1}C_1 \dot\cup \dots \dot\cup \phi^{-1}C_k .$$  We set $C:=C_l$, where $l \in \{ 1,\dots ,k  \}$ is such that $\mu X = \mu \phi^{-1}(C_l )$, and we replace $X$ by $\phi^{-1}(C)$ and $\phi$ by $\phi |_{\phi^{-1}C}$.  Then $\tau_C \circ \phi$ is an isomorphism  $X \to \tau_C C$. 
Let $\mathcal{D}$ be a decomposition of $R^n$ into cells partitioning each $$X\cap \phi^{-1} (\tau^{-1}_C (B_{ij} \cap \tau_C C)).$$
Then $$X=_0 D_1 \dot\cup \dots \dot\cup D_m ,$$ where each $D_i$ is an open cell from $\mathcal{D}$.  Let $D:= D_l$ for $l\in \{ 1,\dots ,m \}$ so that $\mu X=\mu D_l$, and let $B:=B_{ij}$ so that $\tau_C \circ \phi (D)\subseteq B_{ij}$.  By Lemma \ref{box}, we can find a box $P\subseteq \tau_D (D)$ with $\mu P>\widetilde{a}$.
Then $$\tau_P P=[0,\epsilon_1 ] \times [0,\epsilon_2 ] \times \dots \times [0,\epsilon_n ],$$ where each $\epsilon_i \in V^{>0}$ and $\mu P = \Pi_{i=1}^{n} \widetilde{ \epsilon_i }$.
Let $\theta \colon [0,1]^n \to R^n$ be given by $\theta (x)=(\epsilon_1 x_1 ,\dots ,\epsilon_n x_n )$.  Then $\theta ([0,1]^n )=\tau_P P$.  

We define another map $\hat{\theta}\colon \tau_B B \to R^n$ by $\hat{\theta}(x)=(\delta_1 x_1 , \dots , \delta_n x_n )$, where $\delta_1 , \dots ,\delta_n \in R^{>0}$ are chosen in such a way that $\det (\hat{\theta})=\frac{1}{\det{\theta}}$, and $\hat{\theta}(\tau_B B) \subseteq V^n$ (this is possible since $\mu B < \mu P$).  Then $\pi \hat{\theta}(
\tau_B B)$ has empty interior.  
However, the map
$$\hat{\theta } \circ \tau_{B} \circ \tau_{C} \circ \phi \circ \tau^{-1}_{D} \circ \tau^{-1}_P \circ \theta$$ is an isomorphism $[0,1]^n \to \hat{\theta}(\tau_B B)$, a contradiction with the theorem being true in the case when $\interior{\pi X}\not=\emptyset$.
\end{proof}

\begin{definition}\label{defmeasureinV}
For a definable set $X\subseteq V^n$ we set $$\mu X := \widetilde{\frac{1}{\det{A}}}\cdot \mu (TX),$$ where $$T: R^n \to R^n : x\mapsto Ax+b$$ is an affine map with affine transformation matrix $A=(a_{ij})$ such that $a_{ij}=\lambda \in V^{>\ma}$ whenever $i=j$, and $a_{ij}=0$ whenever $i\not= j$, $b\in V^{n}$, and $AX \subseteq [0,1]^n$.
\end{definition}
The next Lemma shows that $\mu X$ is well-defined on $SB[n]$.

\begin{lemma}
Let $X\subseteq V^n$ be definable, and let $$T:R^n \to R^n : x\mapsto Ax+b \mbox{ and } T':R^n \to R^n : x \mapsto A'x+b'$$ be affine transformations for $X$ as in Definition \ref{defmeasureinV}.  Then $$  \widetilde{\frac{1}{\det{A}}} \cdot \mu (TX)=  \widetilde{\frac{1}{\det{A'}}} \cdot \mu (T'X) .$$
\end{lemma}

\begin{proof}
Note that $\interior{\pi X}= \emptyset$ iff $\interior{\pi (TX)}=\emptyset$, and the lemma holds whenever $\interior{\pi X}\not=\emptyset$, since it holds in $\mathbb{R}_0$.  So we may assume $\interior{\pi X}=\emptyset$, in which case
$$\widetilde{\frac{1}{\det{A}}}\cdot \mu TX=\mu TX\mbox{ and }\widetilde{\frac{1}{\det{A'}}}\cdot \mu T'X=\mu T'X,$$
so it suffices to show that $\mu (TX) = \mu (T'X)$.  We set $Y:=TX$ and $S:=T'\circ T^{-1}$.  Then $Y,SY \subseteq [0,1]^n$, and $S$ is an affine transformation with diagonal affine transformation matrix so that each entry on the diagonal is a fixed $\lambda \in V^{>\ma}$.  

To see that $\mu Y \leq \mu SY$, let $a\in \ma^{>0}$ be such that $\widetilde{a}<\mu Y$.  We can find a cell $C\subseteq Y$ and a box $B\subseteq \tau_C C$ with $\widetilde{a} < \mu B$.  But then $SC\subseteq SY$ is a also a cell, and $SB\subseteq \tau_{SC} SC$ is a box such that $\widetilde{a}<\mu SB$.

The inequality $\mu SY \leq \mu Y$ follows by a similar argument when considering $S^{-1}\colon SY \to Y$ instead of $S$.
\end{proof}

\begin{corollary}\label{invariance}
Let $X,Y \subseteq V^n$ be definable and let $\phi \colon X\to Y$ be an isomorphism.  Then $\mu X=\mu Y$.
\end{corollary}

\begin{proof}
Since $\mu$ is invariant under translations, we may assume that $X,Y\subseteq [0,m]^n$.  Let $\theta \colon R^n \to R^n$ be given by $\theta (x) = (\frac{1}{m}x_1 , \dots ,\frac{1}{m}x_n )$.  Then $$\theta |_{Y} \circ \phi \circ \theta^{-1}|_{\theta X} : \theta X \to \theta Y$$ is an isomorphism between subsets of $[0,1]^n$, hence by Theorem \ref{inv}, $\mu \theta X = \mu \theta Y$, and so $\mu X = \mu Y$ by the definition of $\mu$.
\end{proof}

\begin{lemma}\label{products}
Let $X\subseteq [0,1]^m$ and $Y \subseteq [0,1]^n$ be definable.  Then $\mu (X\times Y) = \mu Y \cdot \mu X$. 
\end{lemma}
\begin{proof}
If $\interior{X}=\emptyset$ or $\interior{Y}=\emptyset$, then the lemma holds trivially, so assume that $\interior{X}$ and $\interior{Y}$ are nonempty.

Note that in the case when $\mu X, \mu Y \in \mathbb{R}^{>0}$, the lemma holds, since then $\mu X$ and $\mu Y$ are just the Lebesgue measures of $\pi X$ and $\pi Y$ respectively.  So suppose $\mu X\not\in \mathbb{R}^{>0}$ or $\mu Y \not\in \mathbb{R}^{>0}$ (and hence $\mu (X\times Y) \not\in \mathbb{R}^{>0}$).

Let $\mathcal{C}$ be a decomposition of $R^{m+n}$ into cells that partitions $X\times Y$.  To see that $\mu (X \times Y) \leq \mu X\cdot \mu Y$, let $C \in \mathcal{C}$ be such that $C\subseteq X\times Y$ and $\mu C=\mu (X\times Y)$.  Let further $a\in V^{>0}$ be so that $\widetilde{a}<\mu (X\times Y)$.  By Lemma \ref{box}, we can find a box $B\subseteq \tau_{C} C$ with $\widetilde{a}<\mu B$.  Then $B=pB\times qB$, where $p\colon R^{m+n}\to R^m$ denotes the projection onto the first $m$ coordinates and $q\colon R^{m+n}\to R^{n}$ is the projection onto the last $n$ coordinates.  By Lemma \ref{multcong}, $\mu B=\mu pB \cdot \mu qB$.  Now $\tau_{pC}^{-1}pB\subseteq X$ and $\mu \tau_{pC}^{-1}pB=\mu pB$, since $\tau_{pC}^{-1}|_{pB}$ is an isomorphism $pB \to \tau^{-1}_{pC}(pB)$.  Hence $\mu pB \leq \mu X$.  We now define a map $\hat{\tau}$ on $qB$.  Suppose $p^{m+n}_{m+k}C=(f_{k},g_{k})$ for $1\leq k\leq n$.  Fix $c\in C$, and let $\hat{f}_{k}= f_{k}(p^{m+n}_{m+k-1}c)$.  Set $\hat{\tau}:=(\tau_1 , \dots ,\tau_{n})$, where $\tau_k (x) = x_k + \hat{f}_k$ for $x\in qB$.  Then $\hat{\tau}qB\subseteq Y$ and, since $\hat{\tau}$ is an isomorphism, $\mu qB = \mu \hat{\tau}qB$, so $\mu qB \leq \mu Y$.  It follows that $\widetilde{a}<\mu X \cdot \mu Y$, hence $\mu (X\times Y) \leq \mu X \cdot \mu Y$.

To see that $\mu X \cdot \mu Y \leq \mu (X\times Y)$, let $a\in V^{>0}$ be such that $\widetilde{a}<\mu X \cdot \mu Y$.  Then we can find $b,c \in V^{>0}$ with $\widetilde{a} \leq \widetilde{b}\cdot \widetilde{c}$ and $\widetilde{b}<\mu X$ and $\widetilde{c}<\mu Y$.  
First, suppose $\mu X\not\in \mathbb{R}^{>0}$ and $\mu Y\not\in \mathbb{R}^{>0}$.  Then we can find cells $C\subseteq X$ and $D\subseteq Y$ such that $\mu C=\mu X$ and $\mu D = \mu Y$.  By Lemma \ref{box},  there are boxes $B\subseteq \tau_{C}C$ and $P\subseteq \tau_{D}D$ so that $\widetilde{b}<\mu B$ and $\widetilde{c}<\mu P$.  Note that $C\times D \subseteq X\times Y$ is a cell.  We have $P\times Q\subseteq \tau_{C\times D}(C\times D)$, and hence $\widetilde{a} < \mu (X\times Y)$ because $\tau_{C\times D}$ is an isomorphism.

Finally, suppose that $\mu X\not\in \mathbb{R}^{>0}$ and $\mu Y \in \mathbb{R}^{>0}$ (the case when $\mu X \in \mathbb{R}^{>0}$ and $\mu Y\not\in \mathbb{R}^{>0}$ is similar).  Proceed as in the previous case, but let $D\subseteq Y$ be any cell so that $\interior{\pi D}\not=\emptyset$, and let $P\subseteq \tau_D D$ be a box so that $\interior{\pi P}\not=\emptyset$.
\end{proof}

We now have the following theorem:
\begin{theorem}\label{maintheorem}
For each $n$,
there is a map $\mu_n \colon SB[n] \to \widetilde{V}$ such that for all $X,Y \in SB[n]$, $\mu_n (X \dot\cup Y) = \mu_n X + \mu_n Y$, and $\mu_n X>0$ iff $\interior{X} \not= \emptyset$.  Furthermore, if $X\in SB[m]$ and $Y \in SB[n]$, then $\mu_{m+n} (X\times Y) = \mu_{m} X \cdot \mu_{n} Y$, and $\mu_n Y=\mu_n \phi (Y)$ whenever $\phi$ is an isomorphism $Y\to \phi (Y)$.
\end{theorem}
\begin{proof}
For a given $n$, let $\mu_n \colon SB[n] \to \widetilde{V}$ be as in Definition \ref{defmeasureinV}.  Finite additivity of $\mu_n$ follows from Theorem \ref{main}.  It follows from Lemma \ref{box} and Theorem \ref{inv} that for $X\in SB[n]$, $\mu_n (X)>0$ implies $\interior{X}\not=\emptyset$.  The reverse implication is immediate from the definition of $\mu_n$.  For $X\in SB[m]$ and $Y \in SB[n]$, $\mu_{m+n}(X\times Y)= \mu_{m}X \cdot \mu_{n}Y$ is implied by Lemma \ref{products}.
Finally, invariance under isomorphisms is Corollary \ref{invariance}.

\end{proof}


\end{section}

\begin{section}{A special case}
In this section, we assume that $R$ is such that for all $x,y \in \ma^{\geq 0}$, $x\sim y$ iff $v(x)=v(y)$.  We modify the definition of $\mu$ to obtain a finitely additive measure $\nu$ on all of $B[n]$, which takes values in the Dedekind completion of $\Gamma$, and is such that $\nu X>0$ iff $\interior{X}\not=\emptyset$.  The price we pay for extending the collection of measurable sets to $B[n]$, is that we need to identify all sets of ``finite, non-infinitesimal size''.  For example, $\nu X=\nu Y$ whenever $X,Y \in SB[n]$ are such that $\pi X$ and $\pi Y$ have non-empty interior.

Note that the condition $ x \sim y$ iff $v(x)=v(y)$, for all $x,y \in \ma^{\geq 0}$, is satisfied when the underlying set of $R$ is the field of Puiseux series
$\bigcup_n \mathbb{R}((t^{\frac{1}{n}}))$ in $t$ over $\mathbb{R}$.  The results of this section thus apply to the L-R field (see \cite{lr}, and Introduction).

\begin{definition}
Let $x,y \in R^{\geq 0}$.  Then $x\approx y$ iff $v(x)=v(y)$.
\end{definition}
We define Dedekind cuts in $R^{\geq 0}/\approx$ analogously to Dedekind cuts in $V^{\geq 0}/\sim$ (see the paragraph above Definition \ref{operations}), and
we let $\widetilde{R}$ be the collection of all Dedekind cuts in $R^{\geq 0}/\approx$.  We define $\leq$ and $+$ and $\cdot$ on $\widetilde{R}$  as in Definition \ref{operations}, with $\widetilde{R}$ in place of $\widetilde{V}$.

The proof of the next lemma is straight-forward and left to the reader. 
\begin{lemma}
The operations $+$ and $\cdot$ are well-defined and make $\widetilde{R}$ into an ordered semiring.
\end{lemma}
For $x\in V^{\geq 0}$ we shall abuse notation by identifying the element $\widetilde{x} \in \widetilde{V}$ with its image in $\widetilde{R}$ under the $(+,\cdot , 0, 1)$-homomorphism induced by the map $$V^{\geq 0}/\sim \; \rightarrow R^{\geq 0}/\approx : [x]_{\sim} \mapsto [x]_{\approx}.$$
\begin{lemma}\label{newcong}
For all $x,y \in R^{\geq 0}$, $\widetilde{x}+\widetilde{y}=\widetilde{x+y}$ and $\widetilde{x}\cdot \widetilde{y}=\widetilde{x\cdot y}$.
\end{lemma}
\begin{proof}
Straight-forward and left to the reader.
\end{proof}
\begin{definition}\label{defmeasureinR}
For $X\in B[n]$ we set $$\nu X := \widetilde{\frac{1}{\det{A}}}\cdot \mu (TX),$$ where $$T: R^n \to R^n : x\mapsto Ax+b$$ is an affine map with a diagonal affine transformation matrix $A=(a_{ij})$ such that $a_{ii}=\lambda \in (0,1]$ for $i=1,\dots ,n$, $b\in R^{n}$, and $TX \subseteq [0,1]^n$.
\end{definition}
The next Lemma shows that $\nu X$ is well-defined.
\begin{lemma}
Let $X\in B[n]$ be definable, and let $$T:R^n \to R^n : x\mapsto Ax+b \mbox{ and } T':R^n \to R^n : x \mapsto A'x+b'$$ be affine transformations for $X$ as in Definition \ref{defmeasureinR}.  Then $$  \widetilde{\frac{1}{\det{A}}} \cdot \mu (TX)=  \widetilde{\frac{1}{\det{A'}}} \cdot \mu (T'X) .$$
\end{lemma}
\begin{proof}
We set $Y=TX$ and $$S=T'\circ T^{-1}\colon [0,1]^n \to [0,1]^n .$$  
Then $S$ is an affine map with diagonal transformation matrix $(a_{ij})$, where $a_{ii}=\alpha \in R^{>0}$ for $i=1,\dots ,n$.  It suffices to show that
$\mu Y = \widetilde{\frac{1}{\alpha^n}}\cdot \mu SY$.  
This is clearly satisfied if $\interior{\pi Y}\not=\emptyset$, since then $\widetilde{\alpha }=\widetilde{1}$.  It also holds in the case when $Y$ is a box.  So if $\mu Y <\widetilde{1}$, then the lemma is implied by Lemma \ref{box} and Lemma \ref{boxlemma}.
\end{proof}
It is now clear that we have an analog of Theorem \ref{maintheorem}:
\begin{theorem}\label{secondthm}
Suppose $R$ is such that for all $x,y \in \ma^{>0}$, $v(x)=v(y)$ iff $x^{q}\leq y \leq x^{p}$ for all $p,q \in \mathbb{Q}^{>0}$ with $p<1<q$.
Then,
for each $n$,
there is a map $\mu_n \colon B[n] \to \widetilde{R}$ such that for all $X,Y \in B[n]$, $\mu_n (X \dot\cup Y) = \mu_n X + \mu_n Y$, and $\mu_n X>0$ iff $\interior{X} \not= \emptyset$.  Furthermore, if $X\in B[m]$ and $Y \in B[n]$, then $$\mu_{m+n} (X\times Y) = \mu_{m} X \cdot \mu_{n} Y,$$  and $\mu_n Y=\mu_n \phi (Y)$ whenever $\phi$ is an isomorphism $Y\to \phi (Y)$.
\end{theorem}

\end{section}



\end{document}